\documentclass[11pt,dvipdfmx]{article}
\usepackage{latexsym,amssymb}
\usepackage{color}
\usepackage{amsmath,amsthm,mathrsfs}
\allowdisplaybreaks

\numberwithin{equation}{section}

\usepackage{bm}
\newcommand{\R}{\bm{R}}
\newcommand{\e}{\varepsilon}
\newcommand{\vphi}{\varphi}

\newcommand{\Pu}{\mathcal{P}}
\newcommand{\la}{\lambda}
\newcommand{\La}{\Lambda}
\newcommand{\tr}{\mathrm{tr}}

\newtheorem{Thm}{\bf Theorem}[section]
\newtheorem{Lemma}{\bf Lemma}[section]
\newtheorem{Prop}{\bf Proposition}[section]

\newtheorem{Remark}{\bf Remark}[section]

\newtheorem{Def}{\bf Definition}[section]

\providecommand{\keywords}[1]{\textbf{\textit{Keywords:}} #1}
\providecommand{\MSC}[1]{\textbf{\textit{MSC:}} #1}

\begin{document}
\title{Existence of global-in-time solutions 
to a system \\
of fully nonlinear parabolic equations}
\author{Takahiro Kosugi \footnote{Tottori University of Environmental Studies, Tottori, Japan, t-kosugi@kankyo-u.ac.jp} 
\,and Ryuichi Sato\footnote{Fukuoka University, Fukuoka, Japan, rsato@fukuoka-u.ac.jp}
\vspace{5pt}\\
}
\date{}

\maketitle

\begin{abstract}
We consider the Cauchy problem for a system of fully nonlinear parabolic equations. 
In this paper, we shall show the existence of global-in-time solutions to the problem. 
Our condition to ensure the global existence is specific to the fully nonlinear parabolic system. 
\end{abstract}
\keywords{viscosity solutions, fully nonlinear parabolic systems, 
global-in-time solutions, comparison principle}\\
\MSC{35A01, 35D40, 35K45, 35K55}

\section{Introduction}
Let us consider the Cauchy problem for a weakly coupled system of nonlinear parabolic equations 
\begin{equation}\label{eq:P}
\left\{
\begin{aligned}
\partial_t u_{1} +F_{1}(x,D^2 u_{1}) = |u_{2}|^{p-1}u_{2}, \quad x\in \R^N, \ t>0, \\
\partial_t u_{2} +F_{2}(x,D^2 u_{2}) = |u_{1}|^{q-1}u_{1}, \quad x\in \R^N, \ t>0,
\end{aligned}
\right.
\end{equation}
with initial condition
\begin{equation}\label{eq:ic}
u_{i}(x,0) = u_{i0}(x), \quad x\in \R^N  \mbox{ for } i=1,2,
\end{equation}
where $N\geq 1$, $p,q>0$, $F_{1}, F_2\in C(\R^{N}\times S^{N})$ are uniformly elliptic and homogeneous of order one, and $u_{10}$, $u_{20} \in BUC(\R^N)$ are nonnegative. 
Here $\partial_{t} u_{i}$ denotes the derivative $\partial u_{i}/\partial t$ and $D^{2}u_{i}$ denotes the Hessian matrix of $u_{i}$ in the variable $x$. 
Throughout this paper, we let $S^{N}$ denote the $N\times N$ real symmetric matrices and let $BUC(\R^N)$ denote the set of bounded uniformly continuous functions on $\R^N$. 

In~\cite{EH1}, Escobedo and Hererro considered the Cauchy problem for a system of semilinear parabolic equations 
\begin{equation}\label{eq:F}
\partial_{t}u_{1} -\bigtriangleup u_{1} = u_{2}^{p}, \quad 
\partial_{t}u_{2} - \bigtriangleup u_{2} = u_{1}^{q}, \quad x\in \R^{N}, t>0 
\end{equation}
with \eqref{eq:ic}, 
where $N\geq1$, $p,q>0$, 
and $\bigtriangleup $ denotes the Laplace operator, that is, 
\[
\bigtriangleup:= \sum_{j=1}^{N}\frac{\partial^{2}}{\partial x_{j}^{2}}. 
\]
The system \eqref{eq:F} agrees with the case 
$F_{1}(x,X)=F_{2}(x,X)= -\tr(X)$ for $x\in \R^{N}$, $X\in S^{N}$. 
Escobedo and Hererro proved that if $pq>1$ and 
\[
\frac{\max\{p,q\}+1}{pq-1} \geq \frac{N}{2}, 
\]
then every nontrivial nonnegative solution to \eqref{eq:F} blows up in a finite time. 
On the other hand, if $pq>1$ and 
\begin{equation}\label{eq:SCGS}
\frac{\max\{p,q\}+1}{pq-1} < \frac{N}{2}, 
\end{equation}
then there exists a global-in-time solution to \eqref{eq:F} for 
some $u_{10}$, $u_{20}$. 
These results show that the existence of nonnegative global-in-time solutions to \eqref{eq:F} is clarified by the curve 
\begin{equation}\label{eq:exp.sys.Fujita}
\frac{\max\{p,q\}+1}{pq-1} = \frac{N}{2}. 
\end{equation}
%
This Fujita type result for \eqref{eq:F} 
is extended by \cite{FI2}, \cite{U} to the case where the system with linear but unequal principal parts. 
In~\cite{U}, 
$-\bigtriangleup u_{1}, -\bigtriangleup u_{2}$ are replaced by the linear operators of the form 
\begin{equation*}
\begin{aligned}
L_{1}u_{1} 
= -\sum_{j,k=1}^{N} \frac{\partial}{\partial x_{j}}
\left(a^{jk}\frac{\partial u_{1}}{\partial x_{k}}\right), \quad
L_{2}u_{2} 
= -\sum_{j,k=1}^{N} \frac{\partial}{\partial x_{j}}
\left(b^{jk}\frac{\partial u_{2}}{\partial x_{k}}\right), 
\end{aligned}
\end{equation*}
where the coefficients $a^{jk}$, $b^{jk}$ are sufficiently smooth, uniformly elliptic and symmetric. 
In particular, the system with constant diffusion coefficients $L_1 = -d_1\bigtriangleup u_{1}$, $L_2 = -d_2\bigtriangleup u_{1}$, $d_1$,$d_2>0$ is considered in~\cite{FI2} (see also~\cite{FI1} for another context). 
The Fujita exponent for the system 
\[
\partial_{t} u_{1} +L_{1} u_{1} = u_{2}^{p}, \quad  
\partial_{t} u_{2} +L_{2} u_{2} = u_{1}^{q}, \quad x\in\R^{N}, \ t>0
\]
is also given by \eqref{eq:exp.sys.Fujita}. 
Namely, the Fujita exponent is given by~\eqref{eq:exp.sys.Fujita} if  
the principal parts are linear.

Let us introduce results for a single equation. 
Setting $u_{1}=u_{2}=u$, $F_{1}=F_{2}=F$ and $p=q>1$, then \eqref{eq:P} becomes a single nonlinear parabolic equation 
\begin{equation}\label{eq:sF}
\partial_{t}u + F(x,D^{2} u) = u^{p}, \quad x\in \R^{N}, \ t>0. 
\end{equation}
Typical examples of $F$ are given below. 
When $F(D^{2}u) = -\bigtriangleup u$, \eqref{eq:sF} is the Fujita equation. 
In~\cite{F}, Fujita considered the Cauchy problem for \eqref{eq:sF} with $F(x,D^{2}u)= -\bigtriangleup u$. 
He proved that the critical exponent for the existence of nonnegative global-in-time solutions is given by 
\[
\frac{1}{p-1} = \frac{N}{2}. 
\]
More precisely, if $1<p < p_{F} := 1+2/N$, then all positive solutions blow-up in a finite time, while if $p>p_{F}$, then there exists a positive global-in-time solution of \eqref{eq:sF}. 
(Readers are referred to \cite{DL} for a survey of blow-up problems.) 
When $F$ is fully nonlinear, 
the critical exponent for the existence of global-in-time solutions to \eqref{eq:sF} was obtained in \cite{MQ1,MQ2}. 
We employ the viscosity solutions to treat fully nonlinear equations. 
To give prrecise examples and state the existence of viscosity solutions, we  suppose precise assumptions on $F_{i}$'s. 
The definition of viscosity solutions is given in the next section. 

For $i=1,2$, we assume that $F_{i}:\R^{N}\times S^{N}\to\R$ satisfies the following properties. 
\begin{enumerate}
\item
$F_{i}$ is continuous in $\R^{N}\times S^{N}$, 
that is, 
\begin{align}\label{eq:Fconti}
F_{i}\in C(\R^{N}\times S^{N}). 
\end{align}
\item
There exist constants $0<\lambda_{i}\leq \Lambda_{i}$ such that
\begin{align}\label{eq:ellipticity}
\Pu_{i}^{-}(X-Y)\leq F_{i}(x,X)-F_{i}(x,Y)\leq \Pu_{i}^{+}(X-Y)
\end{align}
for $(x,X,Y)\in \R^{N}\times S^{N}\times S^{N}$, 
where $\Pu^{\pm}_{i}$ are the Pucci extremal operators 
defined by 
\begin{align*}
\Pu^{+}_{i}(X)
&=\Pu^{+}_{\la_{i},\La_{i}}(X)
:=\max\{\tr[-A X] \ | \la_{i} I\leq A\leq \La_{i} I,\ A\in S^{N}\},\\
\Pu^{-}_{i}(X)
&=\Pu^{-}_{\la_{i},\La_{i}}(X)
:=\min\{\tr[-A X] \ | \la_{i} I\leq A\leq \La_{i} I,\ A\in S^{N}\}, 
\end{align*}
for $X\in S^{N}$.  

\item
$F_{i}$ is Lipshitz continuous in $x$. 
Namely, there exists $L>0$ such that 
\begin{equation}\label{eq:FLip}
|F_{i}(x,X)- F_{i}(y,X)| \leq L(\|X\| + 1)|x-y|
\end{equation}
for all $X\in S^{N}$ and $x,y \in \R^{N}$. 
Here $\|X\|$ stands for the operator norm of $X$. 
\item
$F_{i}$ is homogeneous of order one. 
Namely, 
\begin{equation}\label{eq:Fhom}
F_{i}(x,\mu X) = \mu F_{i}(x, X)
\end{equation}
for $\mu \geq0$, $ x\in \R^{N}$, $ X\in S^{N}$. 
\end{enumerate}
We shall give two examples of $F:\R^{N}\times S^{N}\to\R$ satisfying above conditions. 
\begin{itemize}
\item
Let $0<\gamma<1$. 
The operator 
\[
F(D^{2}u) = 
\max\left\{ -\frac{\bigtriangleup u}{1-\gamma}, -\frac{\bigtriangleup u}{1+\gamma} \right\} 
\]
is nonlinear and convex. 
Equation $\partial_t u +F(D^2 u)=0$ is called the Barenblatt equation of Elasto-Plastic equation. 
See~\cite{HV} and \cite{KPV}. 
\item
Let $N=2$. 
Then 
\[
F(D^{2}u) = 
\min\left\{\max\{ -\bigtriangleup u, -2\bigtriangleup u\}, - u_{x_{1}x_{1}}-2 u_{x_{2}x_{2}}\right\} 
\]
is a nonlinear and nonconvex operator. 
\end{itemize}

We now state the comparison principle. 
The proof is given in Section~\ref{sec:CP}. 
\begin{Thm}[Comparison principle]\label{thm:CP}
Assume that $p,q\geq 1$ and let $T>0$. 
Let $(u_{1},u_{2})\in USC\cap L^{\infty}(\R^{N}\times [0,T))^{2}$ 
be a viscosity subsolution and $(v_{1},v_{2})\in LSC\cap L^{\infty}(\R^{N}\times [0,T))^{2}$ be a viscosity supersolution  of \eqref{eq:P}, respectively. 
If 
\begin{align*}
u_{i}(\cdot,0)\leq v_{i}(\cdot,0)\quad \mbox{in } \R^{N} \mbox{ for }i=1,2,
\end{align*}
then 
\begin{align*}
u_{i}\leq v_{i}\quad \mbox{in } \R^{N}\times (0,T) \mbox{ for }i=1,2. 
\end{align*}
\end{Thm}

Existence of viscosity solutions to \eqref{eq:P} and \eqref{eq:ic} is guaranteed by the following: 
\begin{Thm}\label{thm:existence }
Assume that $p,q\geq1$ and $pq>1$. 
Let $u_{10}, u_{20}\in BUC(\R^{N})$. 
There exist $T>0$ and a unique viscosity solution $(u_{1},u_{2})$ of \eqref{eq:P} satisfying \eqref{eq:ic} in $\R^N \times[0,T]$.  
Furthermore, if $u_{i0}\geq0$ for $i=1,2$, then $u_{i}\geq0$ for $i=1,2$, as long as the solution exists. 
Moreover, $u_i \in BUC(\R^N\times[0,T))$. 
\end{Thm}
We temporarily go back to the results obtained by Meneses and Quaas~\cite{MQ1,MQ2}. 
They treated \eqref{eq:sF} for the case $F=F(D^{2}u)$ is an $x$-independent operator which satisfies \eqref{eq:Fconti}--\eqref{eq:Fhom} with the ellipticity constants 
$0<\lambda\leq \Lambda$. 
Then there exists $\alpha=\alpha(F)>0$ such that 
if $1<p \leq 1+1/\alpha$, then there exists no global-in-time solutions for any $u_{0}\in BUC(\R^{N})$. 
While if $p> 1+1/\alpha$, then there exists a global-in-time solution for some $u_{0}\in BUC(\R^{N})$. 
These results mean that $1+1/\alpha$ is the Fujita exponent for \eqref{eq:sF}. 
\begin{Remark}
We give several remarks about $\alpha=\alpha(F)$. 
\begin{enumerate}
\item
If $F$ and $G$ are uniformly elliptic homogeneous such that $F\leq G$, then
\[
\alpha(F) \leq \alpha(G). 
\]
Moreover, if $0<\lambda \leq \Lambda$ are the ellipticity constants of $F$, 
then it holds that 
\[
\frac{N\lambda}{2\Lambda} 
\leq
\alpha(\Pu_{\lambda,\Lambda}^{-}) 
\leq \alpha(F) 
\leq 
\alpha(\Pu_{\lambda,\Lambda}^{+})
\leq
\frac{N\Lambda }{2\lambda}. 
\]
See \cite[(3.21)]{AT} and \cite[Lemma~2.2]{MQ2}. 
\item
If $F=F(X)$ is convex, then 
$ \alpha(F) \geq N/2$ 
and this inequality is strict unless $F$ is linear (see \cite[Example~3.12]{AT}). 
\item
Note that $\alpha$ coincides with the eigenvalue of 
\begin{equation}\label{eq:EVP1}
F(D^{2}\psi) - \frac{1}{2}y\cdot \nabla \psi = \alpha \psi, \quad y\in \R^{N}. 
\end{equation}
\end{enumerate}
\end{Remark}
We next give a remark for the case $F=F(x,D^{2}u)$ depending on $x$. 
In~\cite{MQ2}, it was shown that there exists $\tilde{\alpha}=\tilde{\alpha}(F)>0$ such that 
for all solutions of 
\begin{equation*}
\partial_{t}w + F(x,D^{2}w)=0, \quad x\in \R^{N}, \ t>0, \quad w(x,0)=w_{0}(x), 
\quad x\in \R^{N}
\end{equation*}
satisfy 
\begin{equation*}
\lim_{t\to\infty} t^{\tilde{\alpha}} \|w(\cdot,t)\|_{L^{\infty}} <\infty, 
\end{equation*}
whenever 
$w_{0}\in BUC(\R^{N})$ satisfies $0\leq w_{0}(x) \leq A\exp(-B|x|^{2})$ for some $A,B>0$. 
While if $\beta>\tilde{\alpha}$, it holds that 
\begin{equation*}
\lim_{t\to\infty} t^{\beta} \|w(\cdot,t)\|_{L^{\infty}} =\infty. 
\end{equation*}
This is well-known for the case $F(D^{2}u)=-\bigtriangleup u$ (see e.g. \cite{Me}). 
For $x$ depending case, the Fujita exponent is also given by $1+1/\tilde{\alpha}$. 
However, the critical case $p=1+1/\tilde{\alpha}$ has never been treated for $x$ depending case even for the single equation. 

In this paper, we would like to prove the existence of global-in-time solutions for the system of fully nonlinear parabolic equations~\eqref{eq:P}. 
In our setting, we have the choices of a combination of $F_{1}$ and $F_{2}$. 
Let $\alpha_{i}=\alpha(F_{i})>0$ be the corresponding eigenvalue of \eqref{eq:EVP1} replacing $F$ by $F_{i}$. 
As mentioned above, $\alpha_{i}>N/2$ if $F_{i}$ is convex and nonlinear. 
Therefore, we can expect a different condition to guarantee the existence of a global-in-time solution from \eqref{eq:SCGS}. 

Our main result is the following: 
\begin{Thm}\label{thm:main}
Let $F_{1}$, $F_{2}$ be independent of $x$. 
Suppoe that $F_{i}$'s satisfy \eqref{eq:Fconti}--\eqref{eq:Fhom}. 
Further assume that $p,q\geq 1$ satisfy $pq>1$ and 
\begin{equation}\label{eq:pqellip}
p > \frac{\Lambda_2}{\lambda_1}, \quad
q > \frac{\Lambda_1}{\lambda_2}. 
\end{equation}
There exist positive constants $\alpha_{1}$ and 
$\alpha_{2}$ such that, if 
\begin{equation}\label{eq:pqineq}
\frac{p+1}{pq-1} < \alpha_{1} \quad \mathit{and} \quad 
\frac{q+1}{pq-1} < \alpha_{2}, 
\end{equation}
then there exists a global-in-time solution to \eqref{eq:P} and \eqref{eq:ic} for some $u_{10}$, $u_{20}\in BUC(\R^N)$. 
\end{Thm}
Our theorem gives a sufficient condition for the exsistence of global-in-time solutions to \eqref{eq:P} similar to the Fujita type result. 
We consider a solution $\psi$ of the problem
\[
F(D^{2}\psi) -\frac{1}{2}y\cdot D\psi = \mu \psi, \quad y\in \R^{N}, \quad
\lim_{|y|\to\infty}\psi(y)=0, 
\]
where $F$ satisfies \eqref{eq:Fconti}--\eqref{eq:Fhom}. 
To prove our main theorem, let us apply an estimate for $\psi$ 
of the form 
\[
c \exp(-\delta|y|^2) \leq \psi(y) \leq C\exp(-\delta|y|^2). 
\]
Using this estimate under the condtion~\eqref{eq:pqellip}, we can find a supersolution to obtain global-in-time solutions. 
See Section~4. 
Note that~\eqref{eq:pqellip} is not needed for the single equation. 
%

The rest of this paper is organized as follows. 
In Section~2, we give a precise definition of viscosity solutions of the Cauchy problem and prepare several notation. 
In Section~3, we give a proof of the existence of local-in-time solutions of \eqref{eq:P} and \eqref{eq:ic}. 
In Section~4, we prove Theorem~\ref{thm:main}. 
We give a proof of Theorem~\ref{thm:CP} in Section~5. 
In the appendix, for the convenience of the reader, we give a detailed proof of Perron's method. 

\section{Preliminaries}

For a real valued function $f$ defined in $\R^N \times(0,T)$, 
define the upper (resp. lower) semi-continuous envelope $f^*$ (resp. $f_*$) of $f$  by 
\begin{equation}\label{usce}
\begin{aligned}
f^*(x,t) := \lim_{\e\to0}\sup_{\substack {y\in B(x,\e) \\ |s-t|<\e}} f(y,s),  \quad 
f_*(x,t) := \lim_{\e\to0}\inf_{\substack {y\in B(x,\e) \\ |s-t|<\e}} f(y,s), 
\end{aligned}
\end{equation}
for $x\in \R^N$, $t\in(0,T)$. 
It is well known that $f^*$ is upper semi-continuous, $f_*$ is lower semi-continuous and 
$f_* \leq f \leq f^*$. 
Furthermore, if $f$ is upper semi-continuous, then $f^*=f$. 
The same property holds for $f_*$.

We prepare some notation. 
Let $A$ be a subset of $\R^{N}\times[0,\infty)$. 
The sets 
$USC(A)$ and $LSC(A)$ stand for the set of upper semicontinuous functions on $A$ and lower semicontinuous functions on $A$ respectively. 
Let $\Omega\subset \R^{N}$ and 
let $BUC(\Omega)$ denote the set of bounded uniformly continuous functions on $\Omega$. 
For $T>0$, $C^{2,1}=C^{2,1}(\R^n\times(0,T))$ denotes the set of all functions which is $C^2$ in the variable $x$  and $C^1$ in the variable $t$.


We recall the definition of viscosity solutions of general parabolic systems
\begin{align}\label{eq:generaleq}
\partial_{t}u_{i}+G_{i}(x,t,u_{1},\ldots,u_{m},Du_{i},D^{2}u_{i})=0, \quad\mbox{in } \Omega\times (0,T),\ \mbox{for } i=1,\ldots,m, 
\end{align}
where  $T\in(0,\infty]$ and $\Omega$ is an open subset of $\R^{N}$.

\begin{Def}\label{dfn:vsol}
We call $u=(u_{1},\ldots,u_{m}):\Omega\times (0,T)\to\R^{m}$ a viscosity subsolution (resp., supersolution) of \eqref{eq:generaleq} if for $(i,x,t,\phi)\in\{1,\ldots,m\}\times\Omega\times (0,T)\times C^{2,1}(\Omega\times (0,T))$,
\[
\begin{aligned}
\partial_{t}\phi(x,t)+G_{i}(x,t,u_{1}^*(x,t),\ldots,u_{m}^*(x,t),D\phi(x,t),D^{2}\phi(x,t))
&\leq0,\\
(resp. \ \partial_{t}\phi(x,t)+G_{i}(x,t,{u_{1}}_*(x,t),\ldots,{u_{m}}_*(x,t),D\phi(x,t),D^{2}\phi(x,t))
&\geq0)\\
\end{aligned}
\]
provided that $u_{i}^*-\phi$ (resp. ${u_{i}}_* -\phi$)  attains its local maximum (resp., minimum) at $(x,t)\in\Omega\times(0,T)$. We call $u:\Omega\times (0,T)\to\R^{m}$ a viscosity solution of \eqref{eq:generaleq} if $u$ is a viscosity sub- and supersolution of \eqref{eq:generaleq}.
\end{Def}
We also define a solution to the Cauchy problem. 
\begin{Def}
Let $u=(u_{1},\ldots,u_{m}):\Omega\times (0,T)\to\R^{m}$ be a viscosity subsolution of \eqref{eq:generaleq}. 
We call $u$ a viscosity subsolution of the Cauchy problem \eqref{eq:generaleq} and 
\[
u_1(\cdot,0)=u_{10},\dots,u_m(\cdot,0)=u_{m0} 
\]
if $u$ satisfies 
\begin{equation*}
u_1(\cdot,0) \leq u_{10},\dots, u_m(\cdot,0) \leq u_{m0} \quad \mathrm{in} \ \Omega. 
\end{equation*}
A viscosity supersolution is also defined in the same way. 
\end{Def}
\begin{Def}
Define parabolic semi-jet $PJ^{2,+}u(x,t)$ of a function 
$u:\R^N\times(0,\infty)\to\R$ at $(x,t)\in\R^N\times[0,\infty)$ 
by 
\begin{equation}\label{eq:jet}
\begin{aligned}
&PJ^{2,+}u(x,t) \\
&:=
\biggl\{
(a,z,X)\in \R\times\R^N\times S^N \ \biggl| \ 
u(y,s) \leq u(x,t) + \langle z,y-x\rangle\\
& \quad+ 
\frac{1}{2}\langle X(y-x),y-x\rangle+ a(s-t) +o(|y-x|^2 +|s-t|) 
\quad \mathrm{as} \ y\to x, s\to t. 
\biggr\}, 
\end{aligned}
\end{equation}
where $\langle\cdot,\cdot\rangle$ denotes the standard innnar product on $\R^N$. 
We also define 
$PJ^{2,-}u(x,t) := -PJ^{2,+}(-u(x,t))$. 
Moreover, a sort of 
closure of semi-jet $\overline{PJ}^{2,\pm}u(x,t)$ is defined as follows: 
$(a,z,X)\in \R\times\R^N\times S^N$ is a point of 
$\overline{PJ}^{2,\pm}u(x,t)$ if there exist 
sequences 
$(x_k,t_k)\in \R^N\times(0,\infty)$ and $(a_k,z_k,X_k)\in PJ^{2,\pm}u(x,t)$ such that 
\[
x_k \to x, \quad t_k\to t, \quad  u(x_k,t_k) \to u(x,t), \quad a_k \to a, 
\quad
z_k \to z, \quad 
X_k \to X 
\]
as $k\to\infty$. 
\end{Def}

\section{Existence of local-in-time solutions}
In this section, we give proof of Theorem~\ref{thm:existence }. 
To prove the local existence of viscosity solution, 
we refer to important results from \cite{CL}. 
The following Lemma is modified for the convenience of our argument. 
\begin{Lemma}\label{lem:CP1}
Let $F:S^N \to\R$ be continuous and let satisfy the ellipticity condition 
\begin{equation}\label{eq:Fellip}
F(Y) \leq F(X) \quad \mathrm{whenever} \ X\leq Y, \quad X, Y\in S^{N}. 
\end{equation}
\begin{enumerate}
\item
If $u_{0}$ is uniformly continuous on $\R^{N}$, then the Cauchy problem 
\begin{equation}\label{eq:homCP}
\partial_{t} u + F(D^{2}u) = 0 \quad \mathrm{in} \ \R^{N}, \quad 
u(\cdot,0) = u_{0} \quad \mathrm{in} \ \R^{N}
\end{equation}
has a unique viscosity solution $u\in C(\R^{N}\times[0,\infty))$, which is uniformly continuous in $x\in \R^{N}$. 
Moreover, if $u_{0} \in BUC(\R^{N})$, 
then the unique solution $u$ of \eqref{eq:homCP} is bounded and uniformly continuous in 
$x\in \R^{N}$. 
\item
Assume $u_{0}\in BUC(\R^{N})$. 
Then the solution $u$ of \eqref{eq:homCP} generates a semigroup 
$\{S(t)\}_{t\geq0}$ on $BUC(\R^{N})$, which satisfies the following properties.  
\begin{enumerate}
\item
For any $\vphi, \psi \in BUC(\R^{N})$, 
\[
\|S(t)\vphi - S(t) \psi\|_{L^{\infty}} \leq \|\vphi - \psi\|_{L^{\infty}},  \quad t>0. 
\]

\item
For any $\vphi\in BUC(\R^{N})$, 
\[
\lim_{t\to +0} \| S(t) \vphi - \vphi\|_{L^{\infty}} = 0. 
\]
\end{enumerate}
\end{enumerate}
\end{Lemma}

\begin{proof}[Proof of Theorem~\ref{thm:existence }]
The proof is based on~\cite{MQ1}. 
It can be seen that $\Pu_{i}^-$ satisfies~\eqref{eq:Fellip} by the definition.  
Let $\{S_{i}(t)\}$ be an order preserving semigroup generated by 
$\Pu_{i}^{-}$. 
Then, by Lemma~\ref{lem:CP1} for each $i=1,2$, $z_{i}(x,t) = [S_{i}(t)u_{i0}](x)$ is a viscosity solution to 
\[
\partial_{t}z_{i} + \Pu_{i}^{-}(D^{2}z_{i})=0, \quad x\in \R^{N}, \, t>0,  \quad
z_{i}(x,0) = u_{i0}(x), \quad x\in \R^{N}. 
\]
Furthermore, $S_{i}(t)$ 
satisfies 
\begin{equation}\label{eq:nonexpansion}
\|S_{i}(t)\vphi - S_{i}(t) \psi\|_{L^{\infty}} \leq \|\vphi - \psi\|_{L^{\infty}}, \quad t>0
\end{equation}
for any $\vphi, \psi \in BUC(\R^{N})$ 
and 
\begin{equation*}
\lim_{t\to +0} \| S_{i}(t) \vphi - \vphi\|_{L^{\infty}} = 0. 
\end{equation*}

Let $T>0$. 
Define~$\Psi: (BUC(\R^{N}\times[0,T])^{2} \to (BUC(\R^{N}\times[0,T])^{2}$ by 
\[
\Psi[v_{1}, v_{2}](t):= (\Phi_{1}[v_{2}](t), \Phi[v_{1}](t)), \quad 0\leq t \leq T, 
\]
where 
\[
\begin{aligned}
\Phi_{1}[v_{2}](t) 
&:= S_{1}(t)u_{10} + \int_{0}^{t} S_{1}(t-s) (|v_{2}(s)|^{p-1}v_{2}(s))\,ds,  \\
\Phi_{2}[v_{1}](t) 
&:= S_{2}(t)u_{20} + \int_{0}^{t} S_{2}(t-s) (|v_{1}(s)|^{q-1}v_{1}(s))\,ds. \\
\end{aligned}
\]
For the sake of convenience, we shall show that $\Psi$ is a contraction on a closed subset of $ (BUC(\R^{N}\times[0,T]))^{2}$. 
For $M>0$ and $T>0$, 
the closed ball $B_{T,M}:= \{v\in BUC(\R^{N}\times[0,T]): \sup_{0\leq t \leq T}\|v(t)\|_{L^{\infty}(\R^{N})} \leq M\}$ is a complete metric space. 
Without loss of generality, we may assume $u_{10}\not\equiv0$. 
Moreover, we only need to consider $\Phi_{1}$ due to the symmetry. 
Set $M:= 2(\|u_{10}\|_{L^{\infty}(\R^{N})}+\|u_{20}\|_{L^{\infty}(\R^{N})})>0$. 
Let $v_{2}, \tilde{v}_{2}\in B_{T,M}$. 
Thanks to \eqref{eq:nonexpansion}, we see that 
$\|S_{i}(t)u_{10}\|_{L^{\infty}(\R^{N})}\leq  \|u_{10}\|_{L^{\infty}(\R^{N})}$ for all $t\in[0,T]$ and so 
$S(t)u_{10}\in B_{M}$. 
Moreover, we have 
\[
\left|\int_{0}^{t} S_{1}(t-s) (|v_{2}(s)|^{p-1}v_{2}(s))\,ds\right|
\leq \int_0^t \||v_{2}(s)|^{p-1}v_{2}(s)\|_{L^{\infty}(\R^N)}
\leq tM^{p}
\]
for $t\in[0,T]$. 
Thus, 
\begin{equation}\label{eq:onto}
\|\Phi_{1}[v_{2}](t)\|_{L^{\infty}(\R^{N})} 
\leq \|u_{10}\|_{L^{\infty}(\R^{N})} + TM^{p}. 
\end{equation}
We next use \eqref{eq:nonexpansion} to see  
\[
\begin{aligned}
|\Phi_{1}[v_{2}](t) - \Phi_{1}[\tilde{v}_{2}](t) |
&\leq
 \int_{0}^{t} 
\|
\{
(|v_{2}(s)|^{p-1}v_{2}(s)) -(|\tilde{v}_{2}(s)|^{p-1}\tilde{v}_{2}(s))
\}
\|_{L^{\infty}(\R^{N})}
\,ds. 
\end{aligned}
\]
By the mean value theorem, we see that there exists some $C>0$ such that 
\[
|\Phi_{1}[v_{2}](t) - \Phi_{1}[\tilde{v}_{2}](t) |
\leq 
CT M^{p-1} \sup_{0\leq t\leq T}\|v_{2}(t) -\tilde{v}_{2}(t)\|_{L^{\infty}(\R^{N})} 
\]
for $t\in[0,T]$. 
It follows that 
\begin{equation}\label{eq:cont}
\sup_{0\leq t\leq T }\|\Phi_{1}[v_{2}]-\Phi_{1}[\tilde{v}_{2}]\|_{L^{\infty}(\R^{N})} 
\leq 
CT M^{p-1} \sup_{0\leq t \leq T}\|v_{2}- \tilde{v}_{2}\|_{L^{\infty}(\R^{N})}. 
\end{equation}
Therefore, taking $T>0$ small enough, we see that $\Psi$ is a contraction map on $(B_{T,M})^{2}$. 
By~\eqref{eq:onto} and~\eqref{eq:cont}, the Banach fixed point theorem is applied and there exists a unique fixed point so that 
$\Psi[v_{1},v_{2}]=(\Phi_{1}[v_{2}],\Phi_{2}[v_{1}]) = (v_{1},v_{2})\in B_{T,M}^{2}$. 
Namely, we have
\begin{equation}\label{eq:FP}
\begin{aligned}
v_{1}(t) 
&= S_{1}(t)u_{10} + \int_{0}^{t} S_{1}(t-s) (|v_{2}(s)|^{p-1}v_{2}(s))\,ds,  \\
v_{2}(t) 
&= S_{2}(t)u_{20} + \int_{0}^{t} S_{2}(t-s) (|v_{1}(s)|^{q-1}v_{1}(s))\,ds \\
\end{aligned}
\end{equation}
in $\R^{N}\times[0,T]$. 
Furthermore, it follows from \eqref{eq:FP} that 
\[
\|v_{1}(t) - S_{1}(t)u_{10}\|_{L^{\infty}(\R^{N})}
\leq M^{p}t \to0
\]
as $t\to+0$, hence 
\[
\lim_{t\to+0}\|v_{1}(t) - u_{10}\|_{L^{\infty}(\R^{N})} =0. 
\]
We have the same convergence of $v_{2}$ by the same argument. 

By the regularity theory (see e.g. 
\cite[Theorem~1.6, Chapter~13]{K}, 
\cite[Theorem~14.10]{L}, 
\cite[Theorem~4.13]{W}), 
we know that $S_{i}(t)u_{i0}$ is a classical solution of 
$\partial_{t} w + \Pu^{-}_{i}(D^{2}w)=0$. 
It follows from \eqref{eq:FP} that $\partial_{t} v_{1}, \partial_{t}v_{2}$ exist. 
Taking the derivative of the right-hand side of \eqref{eq:FP}, we see that 
\[
\begin{aligned}
&\frac{\partial}{\partial t} [S_{1}(t) u_{10}]
= - \Pu_{1}^{-}(D^{2}[S_{1}(t)u_{10}]), \\
&\frac{\partial}{\partial t} 
\int_{0}^{t}S_{1}(t-s)(|v_{2}(s)|^{p-1}v(s))\,ds\\ 
&=
-\int_{0}^{t} \Pu^{-}_{1} (D^{2}[S_{1}(t-s) (|v_{2}(s)|^{p-1}v_{2}(s))])\,ds 
+ |v_{2}(t)|^{p-1}v_{2}(t). 
\end{aligned}
\]
The same estimate also allows $\partial_{t}v_{2}$ to exist. 
Thus, they satisfy 
\[
\begin{aligned}
\partial_{t} v_{1} + \Pu^{-}_{1}(D^{2}S_{1}(t)u_{10}) 
&= -\int_{0}^{t} \Pu^{-}_{1} (D^{2}[S_{1}(t-s) (|v_{2}(s)|^{p-1}v_{2}(s))])\,ds 
+ |v_{2}|^{p-1}v_{2}(t), \\
\partial_{t} v_{2} + \Pu^{-}_{2}(D^{2}S_{2}(t)u_{20}) 
&= -\int_{0}^{t} \Pu^{-}_{2} (D^{2}[S_{2}(t-s) (|v_{1}(s)|^{q-1}v_{1}(s))])\,ds 
+ |v_{1}|^{q-1}v_{1}(t). \\
\end{aligned}
\]
It follows from the property $\Pu_{i}^{-}(X+Y) \geq \Pu_{i}^{-}(X)+\Pu_{i}^{-}(Y)$, that 
\[
\begin{aligned}
&-\int_{0}^{t} \Pu^{-}_{1} (D^{2}[S_{1}(t-s) (|v_{2}(s)|^{p-1}v_{2}(s))])\,ds \\
&\geq 
-\Pu^{-}_{1} \left(
\int_{0}^{t} D^{2}[S_{1}(t-s) (|v_{2}(s)|^{p-1}v_{2}(s))])\,ds
\right). 
\end{aligned}
\]
Furthermore, since $S_{1}(t-s) (|v_{2}(s)|^{p-1}v_{2}(s))$ is of $C^{2}$ as a function of $x$, we have
\[
\int_{0}^{t} D^{2}[S_{1}(t-s) (|v_{2}(s)|^{p-1}v_{2}(s))]\,ds
= D^{2}\int_{0}^{t} S_{1}(t-s) (|v_{2}(s)|^{p-1}v_{2}(s))\,ds, 
\]
hence 
\begin{equation}\label{eq:intineq}
\begin{aligned}
\partial_{t} v_{1} 
+ \Pu_{1}^{-}
\left(
D^{2} 
\left[ 
S_{1}(t)u_{10} + \int_{0}^{t}S_{1}(t-s)|v_{2}|^{p-1}v_{2}(s)\,ds
\right]
\right)
&\geq 
 |v_{2}|^{p-1}v_{2}(t)
, \\
\partial_{t} v_{2} 
+ \Pu_{2}^{-}
\left(
D^{2} 
\left[ 
S_{2}(t)u_{20}+\int_{0}^{t}S_{2}(t-s)|v_{1}|^{q-1}v_{1}(s)\,ds
\right]
\right)
&\geq 
 |v_{1}|^{q-1}v_{1}(t)
\quad
\end{aligned}
\end{equation}
in $\R^{N}\times[0,T]$. 
Note that the integral preserve linear properties. 
We note that for $i=1,2$, 
$F_{i}$ satisfies 
\[
\Pu^{-}_{i}(X) \leq F_{i}(x,X), \quad x\in \R^{N}, \, X\in S^{N}. 
\]
Consequently, we then deduce from~\eqref{eq:intineq} that 
\begin{equation*}
\begin{aligned}
\partial_{t} v_{1} 
+ F_{1}
\left(
x, D^{2}v_{1}
\right)
\geq 
 |v_{2}|^{p-1}v_{2}
, \quad 
\partial_{t} v_{2} 
+ F_{2}
\left(
x, D^{2}v_{2}
\right)
\geq 
 |v_{1}|^{q-1}v_{1}
  \\
\end{aligned}
\end{equation*}
for $x\in \R^{N}$ and $t>0$. 
Namely, 
$(v_{1}, v_{2})$ is a viscosity supersolution of \eqref{eq:P} and \eqref{eq:ic}. 
Replacing $\Pu^{-}$ by $\Pu^{+}$, we can also obtain a viscosity subsolution of \eqref{eq:P} satisfying \eqref{eq:ic}.


By the Perron method and the comparison principle, there exists a continuous viscosity solution $(u_{1},u_{2})$ of \eqref{eq:P} satisfying \eqref{eq:ic}. 
Nonnegativity of solutions follows from the comparison principle. 

Finally, we shall show that $u_i\in BUC(\R^N\times[0,T))$ for $i=1, 2$. 
We refer to~\cite[Section~3.5]{G} for the method. 
Let $(\underline{u}_1,\underline{u}_2)$ and $(\overline{u}_1,\overline{u}_2)$ be 
 viscosity subsolution and viscosity supersolution to \eqref{eq:P} and \eqref{eq:ic} obtained above. 
We can see that 
\begin{equation*}
\underline{u}_i(x,t) - u_{i0}(x) 
\leq 
u_i(x,t) - u_{i0}(x) 
\leq 
\overline{u}_i(x,t) - u_{i0}(x) 
\end{equation*}
for $x,y\in\R^N$, $t>0$. 
There exists a modulus of continuity $\omega:[0,\infty)\to[0,\infty)$, 
$\omega(0)=0$, since $\underline{u}_{i}, \overline{u}_i\in BUC(\R^N\times[0,T))$. 
Then we have
\[
\underline{u}_i(x,t) - u_{i0}(x) 
\geq -\omega(t) 
\]
and 
\[
\overline{u}_i(x,t) - u_{i0}(x) 
\leq \omega(t), 
\]
which implies that 
\begin{equation}\label{eq:UCIC}
\sup_{\substack{{i=1,2}\\x\in\R^n}}|u_i(x,t)-u_{i0}(x)| \leq \omega(t). 
\end{equation}
For fixed $h>0$, set 
\[
\begin{aligned}
\overline{v}_i(x,t) &:= u_i(x,t+h) + \omega(h), \quad x\in\R^N, \ t\geq 0, \\
\underline{v}_i(x,t) &:= u_i(x,t+h) - \omega(h), \quad x\in\R^N, \ t\geq 0, 
\end{aligned}
\]
for $i=1,2$. 
Then $(\overline{v}_1,\overline{v}_2)$ is a viscosity supersolution and 
$(\underline{v}_1,\underline{v}_2)$ is a viscosity subsolution to \eqref{eq:P} and \eqref{eq:ic}. 
We see from~\eqref{eq:UCIC} that $\underline{v}_i(x,0) \leq u_{i0}(x) \leq \overline{v}_i(x,0)$. 
By Theorem~\ref{thm:CP}, we see that
\[
\underline{v}_i(x,t)
\leq
u_i(x,t) 
\leq 
\overline{v}_i(x,t) 
\quad 
x\in \R^N, \ t\in[0,T]. 
\]
Therefore, we obtain
\[
|u_i(x,t) - u_i(x,t+h)| \leq \omega(h) 
\]
for all $x\in\R^N$. 
This shows that $u_i$ is uniformly continuous with respect to variable $t$. 
Since $u_{i0}\in BUC(\R^N)$, there exists another modulus of continuity $\omega$ so that 
\[
\sup_{i=1,2} |u_{i0}(x) - u_{i0}(y)| 
\leq
\omega(|x-y|) 
\]
for $x,y\in\R^N$. 
Similarly to the above discussion, set for $h\in\R^N$, 
\[
\begin{aligned}
\overline{w}_i(x,t) := u_i(x+h,t) + \omega(|h|), \quad x\in\R^N, \ t\geq0, \\
\underline{w}_i(x,t) := u_i(x+h,t) - \omega(|h|), \quad x\in\R^N, \ t\geq0, \\
\end{aligned}
\]
Then $(\underline{w}_1,\underline{w}_2)$ is a viscosity subsolution and 
$(\overline{w}_1,\overline{w}_2)$ is a
supersolution to \eqref{eq:P} and \eqref{eq:ic}. 
Since 
$\underline{w}_i(x,0) \leq u_i(x,0)\leq \overline{w}_i(x,0)$, 
by Theorem~\ref{thm:CP}, we obtain 
\[
\sup_{i=1,2}|u_i(x,t) - u_i(x+h,t)| \leq \omega(|h|) 
\]
for all $t\in[0,T)$. 
Summarizing, $u_i$'s are uniformly continuous in $\R^N\times[0,T)$. 
\end{proof}
%

\section{Existence of global-in-time solutions (proof of Theorem~\ref{thm:main}) }\label{sec:Existence}
In this section, we shall prove the existence of global-in-time solutions to \eqref{eq:P} and \eqref{eq:ic}. 
We use the following Lemma. 

\begin{Lemma}[{\cite[Lemma~3.10]{AT}}]\label{lem:est.profile}
Let $0<\lambda \leq \Lambda$. 
Assume that $F$ satisfies \eqref{eq:ellipticity} and \eqref{eq:Fhom}. 
For each $\delta<(4\Lambda)^{-1}$, there exists $C>0$ such that 
\[
\psi(y) \leq C\exp(-\delta|y|^{2}), \quad y\in\R^{N},
\]
where $\psi$ is the profile function of a unique positive self-similar solution $\Phi$ of $\partial_{t}u +F(D^{2}u)=0$ 
 appearing as $\Phi(x,t)=t^{-\alpha(F)}\psi(x/\sqrt{t})$. 
Likewise, for each $\delta>(4\lambda)^{-1}$, there exists $C>0$ such that 
\[
C\exp(-\delta|y|^{2}) \leq \psi(y), \quad y\in \R^{N}. 
\]
\end{Lemma}

\begin{proof}[Proof of Theorem~1.2]
For $i=1,2$, let $\psi_{i}$ be a positive solution of the eigenvalue problem 
\begin{equation}\label{eq:EP}
F_{i}(D^{2}\psi_{i}) -\frac{1}{2}y\cdot D\psi_{i} = \mu \psi_{i}, \quad y\in \R^{N}, \quad
\lim_{|y|\to\infty}\psi(y)=0. 
\end{equation}
Let $(\alpha(F_{i}),\psi_{i})$ be the eigenpair of~\eqref{eq:EP}. 
Set $\alpha_i = \alpha(F_i)$. 
The existence of solution to \eqref{eq:EP} is obtained in \cite[Section~3]{AT}. 
See also \cite{MQ1}. 
Let us look forward a supersolution to \eqref{eq:P} of the form 
\begin{equation*}
\overline{u_{1}}(x,t):= \e(t+1)^{a}\phi_{1}(x,t+1), \quad
\overline{u_{2}}(x,t):=\tilde{\e}(t+1)^{b}\phi_{2}(x,t+1), 
\end{equation*}
where $\phi_{i}$ is defined by 
\[
\phi_{i}(x,t):= t^{-\alpha_{i}}\psi_{i}(t^{-\frac{1}{2}}x), \quad x\in \R^{N}, \, t>0 
\]
and $a,b$ will be defined. 
We refer to~\cite[Section~32]{QS1} for the case of linear diffusion. 
For each $i=1,2$, $\phi_{i}$ satisfies 
\begin{equation}\label{eq:phii}
\partial_{t} \phi_{i} + F_{i}(D^{2}\phi_{i})=0 \quad 
\mathrm{ in} \  \R^{N}\times(0,\infty)
\end{equation}
 in the sense of viscosity solution. 
In fact, by the argument used in \cite[Lemma~3.1]{MQ1},  
we can see that $\phi_i$ satisfies \eqref{eq:phii}. 

In what follows, we shall find a sufficient condition that $(\overline{u}_1,\overline{u}_2)$ becomes a viscosity supersolution of~\eqref{eq:P}. 
We have 
\begin{equation}\label{eq:u1t}
\begin{aligned}
\partial_{t} \overline{u_{1}} 
&= a \e(t+1)^{a-1}\phi_{1}
+ \e (t+1)^{a} \partial_{t}\phi_{1}(x,t+1), \quad x\in \R^N, \ t\geq0. 
\end{aligned}
\end{equation}
Assume that $\overline{u}_1 -\Phi$ attains its minimum at $(x,t)$ and satisfies that 
\[
(\overline{u}_1-\Phi)(x,t) = 0. 
\]
Note that the function 
\[
\phi_1(\cdot,\cdot+1) - \frac{1}{\e(t+1)^a}\Phi
\]
also attains minimum at $(x,t)$. 
Since $\phi_1(\cdot,\cdot+1)\in C^{1,1}(\R^N\times(0,\infty))$ and 
$\phi_1$ is a viscosity solution of~\eqref{eq:phii}, it holds that 
\begin{equation*}
\begin{aligned}
\partial_t \Phi(x,t) 
&= a\e(t+1)^{a-1}\phi_1 + \e(t+1)^a\partial_t\phi_1\\
&\geq a\e(t+1)^{a-1}\phi_1 - F_1\left(\frac{1}{\e(t+1)^a}D^2\Phi\right) \\
& = \e(t+1)^{a-1}\phi_1 - F_1(D^2\Phi). 
\end{aligned}
\end{equation*}
On the other hands, we have 
\begin{equation}\label{eq:u1p}
\begin{aligned}
\overline{u_{2}}^{p} &= \tilde{\e}^{p} (t+1)^{bp} \phi_{2}^{p} , \quad t\geq0. 
\end{aligned}
\end{equation}
Combining~\eqref{eq:u1t}--\eqref{eq:u1p}, we see that 
\begin{equation}\label{eq:u1ineq}
\partial_{t}\Phi + F_{1}(D^{2}\Phi) - \overline{u_{2}}^{p}
\geq  (t+1)^{bp-\alpha_{2}p}
[\e a(t+1)^{a-1-\alpha_{1}-bp + \alpha_{2}p} \psi_{1} -\tilde{\e}^{p}\psi_{2}^{p}] 
\end{equation}
at $(x,t)$. 
In the same way, it holds that 
 \begin{equation}\label{eq:u2ineq}
\partial_{t}\Phi + F_{2}(D^{2}\Phi) - \overline{u_{1}}^{q}
\geq  (t+1)^{aq-\alpha_{1}q}
[\e b(t+1)^{b-1-\alpha_{1}-aq + \alpha_{1}q} \psi_{2} -\tilde{\e}^{q}\psi_{1}^{q}] 
\end{equation}
at $(x,t)$. 


To ensure that the right-hand sides of~\eqref{eq:u1ineq},~\eqref{eq:u2ineq} become nonnegative, it suffices that 
\[
\begin{aligned}
&a-1-\alpha_{1} - bp + \alpha_{2}p\geq0, \quad 
\e a \psi_{1} \geq \tilde{\e}^{p}\psi_{2}^{p},\\
&b-1-\alpha_{2} -aq + \alpha_{1}q \geq0, \quad 
\tilde{\e}b\psi_{2}\geq \e^{q}\psi_{1}^{q}. 
\end{aligned}
\]
Solving
\[
a-1-\alpha_{1} - bp + \alpha_{2}p=0, \quad 
b-1-\alpha_{2} -aq + \alpha_{1}q =0, 
\]
we find the conditions
\begin{equation}\label{eq:ab}
a=\alpha_{1} - \frac{p+1}{pq-1}, \quad b= \alpha_{2} - \frac{q+1}{pq-1}. 
\end{equation}
Under the conditions~\eqref{eq:ab}, 
$a>0$ and $b>0$ are equivalent to~\eqref{eq:pqineq}, that is, 
\[
\alpha_{1}>\frac{p+1}{pq-1}, \quad 
\alpha_{2}>\frac{q+1}{pq-1}. 
\]
If $\e=\tilde{\e}$, then we can find $\e>0$ so small that 
$a \geq \e^{p-1}(\psi_{2}^{p}/\psi_{1})$ and 
$b \geq \e^{q-1}(\psi_{1}^{q}/\psi_{2})$. 
Note that $\psi_{2}^{p}/\psi_{1}$ and $\psi_{1}^{q}/\psi_{2}$ are bounded. 
Indeed, applying Lemma~\ref{lem:est.profile} for $F_{i}$, 
for each $a_{i}<(4\Lambda_{i})^{-1}$, there exists $C_{i}^{+}>0$ such that 
\[
\psi_{i}(y) \leq C_{i}^{+}\exp(-a_{i}|y|^{2}), \quad y\in \R^{N},\quad i=1,2. 
\]
We also see that for each $b_{i}>(4\lambda)^{-1}$, there exists $C_{i}^{-}>0$ such that 
\[
C_{i}^{-} \exp(-b_{i}|y|^{2}) \leq \psi_{i}(y), \quad y\in \R^{N}, 
\]
hence
\[
C_{1} \frac{\exp(-b_{2}|y|^{2})^{p}}{\exp(-a_{1}|y|^{2})}
\leq
\frac{\psi_{2}^{p}}{\psi_{1}} 
\leq C_{2}\frac{\exp(- a_{2}|y|^{2})^{p}}{\exp(-b_{1}|y|^{2})}, 
\]
where $C_{1}= (C_{2}^{-})^{p}/C_{1}^{+}$ 
and $C_{2}= (C_{2}^{+})^{p}/C_{1}^{-}$. 
We also have a similar estimate for $\psi_{1}^{q}/\psi_{2}$. 
Finally, it follows from \eqref{eq:pqellip} that $\psi_{1}^{q}/\psi_{2}$ and $\psi_{1}^{q}/\psi_{2}$ are bounded. 
Therefore, assuming \eqref{eq:pqineq} and choosing 
$u_{i0}(x):=\bar{u}_{i}(x,0)$, 
by the Perron method, there exists a global-in-time solution $(u_{1},u_{2})$ to \eqref{eq:P} and \eqref{eq:ic}. 
\end{proof}
%

\section{Proof of Theorem~\ref{thm:CP}}\label{sec:CP}
%

In this section, we prove Theorem~\ref{thm:CP}. 
\begin{Lemma}[{\cite[Proposition~3.8 (2)]{Ko}}]\label{lem:structure}
Asssume that $F$ satisfy \eqref{eq:ellipticity} and \eqref{eq:FLip}. 
There exists a modulus of continuity $\omega_{F}:[0,\infty)\to\R$ 
such that, if 
$X$, $Y\in S^N$, $\mu>1$ satisfy 
\begin{equation*}
\begin{aligned}
-3\mu
\begin{pmatrix}
I & O\\
O & I
\end{pmatrix}
\leq 
\begin{pmatrix}
X & O\\
O & -Y
\end{pmatrix}
\leq 
3\mu
\begin{pmatrix}
I & -I\\
-I & I
\end{pmatrix},  
\end{aligned}
\end{equation*}
then 
\[
F (y,Y) - F(x,X)
\leq 
\omega_{F} 
\left(
|x-y| + \mu|x-y|^2
 \right) 
\]
for all $x,y\in\R^N$. 
\end{Lemma}
\begin{Lemma}\label{lem:CV}
Let $(u_{1},u_{2})\in USC\cap L^{\infty}(\R^{N}\times [0,T))^{2}$ (resp., $LSC\cap L^{\infty}(\R^{N}\times [0,T))^{2}$) be a viscosity subsolution (resp., supersolution) of \eqref{eq:P}. 
We set for $i=1,2$, 
\begin{align*}
w_{i}:=e^{-\nu t}u_{i}, 
\end{align*}
where $\lambda>0$ is a constant. 
Then, $(w_{1},w_{2})$ is a viscosity subsolution (resp., supersolution) of 
\begin{equation}\label{eq:lem4.1}
\left\{
\begin{aligned}
\partial_t w_{1} +F_{1}(x,D^2 w_{1}) + \nu w_{1} 
- e^{(p-1)\nu t} |w_{2}|^{p-1}w_{2}=0, \quad x\in \R^N, \ t>0, \\
\partial_t w_{2} +F_{2}(x,D^2 w_{2})+\nu w_{2} 
- e^{(q-1)\nu t} |w_{1}|^{q-1}w_{1}=0, \quad x\in \R^N, \ t>0.
\end{aligned}
\right.
\end{equation}
\end{Lemma}
\begin{proof}
We shall argue only $w_1$. 
Let $\vphi \in C^2(\R^N\times[0,T))$ be such that 
$w_1 -\vphi$ achieve a maximum at $(x_0,t_0)$ and 
\begin{equation*}
(w_1-\vphi)(x_0,t_0)=0. 
\end{equation*}
Then for all $(x,t)\in \R^N\times[0,T)$, 
\[
e^{-\nu t}u_1(x,t)-\vphi(x,t) 
= (w_1 - \vphi)(x,t) \leq (w_1 -\vphi)(x_0,t_0) 
=0. 
\]
We have $u_1(x,t) - e^{-\nu t}\vphi(x,t) \leq0$ for all $(x,t)\in \R^N\times[0,T)$. 
On the other hand, $u_1(x_0,t_0) - e^{\nu t_0}\vphi(x_0,t_0)=0$. 
Thus, $u_1 - e^{\nu t }\vphi$ attains a maximum at $(x_0,t_0)$. 
Since $(u_1,u_2)$ is a viscosity subsolution of \eqref{eq:P}, 
we have, at $(x,t)=(x_0, t_0)$, 
\[
\begin{aligned}
0
&\geq 
\nu e^{\nu t}\vphi + e^{\nu t} \partial_t \vphi + F_1(x,e^{\nu t}D^2\vphi) - |u_2|^{p-1}u_2 \\
&=
\nu e^{\nu t}\vphi + e^{\nu t} \partial_t \vphi 
+ e^{\nu t} F_1(x,D^2\vphi) - e^{p\nu t}|w_2|^{p-1}w_2. \\
\end{aligned}
\]
We here used \eqref{eq:Fhom}. 
Therefore, we obtain 
\[
0\geq \nu w_1 + \partial_t\vphi 
+ F_1(x,D^2\vphi) - e^{(p-1)\nu t} |w_2|^{p-1}w_2. 
\]
By the same argument, we also obtain 
\[
0\geq \nu w_2 + \partial_t\vphi 
+ F_2(x,D^2\vphi) - e^{(q-1)\nu t} |w_1|^{q-1}w_1. 
\]
Consequently, $(w_1,w_2)$ is a viscosity solution of \eqref{eq:lem4.1}. 
\end{proof}
Theorem~\ref{thm:CP} is shown by proving the following proposition: 
\begin{Prop}
Let $p,q\geq 1$ and $T>0$. 
Let $(u_{1},u_{2})\in USC\cap L^{\infty}(\R^{N}\times [0,T))^{2}$ be a viscosity subsolution and $(v_{1},v_{2})\in LSC\cap L^{\infty}(\R^{N}\times [0,T))^{2}$ be a viscosity supersolution  of \eqref{eq:lem4.1}, respectively. 
Assume that there exists a constant $R>0$ such that 
for all $\lambda>0$, 
\begin{equation}\label{eq:prop4.2}
|e^{\nu t}u_i|, |e^{\nu t}v_i| \leq R \quad \mathrm{in} \ \R^N\times[0,T), i=1,2.
\end{equation}
If $u_{i0} \leq v_{i0}$ in $\R^N$ for $i=1,2$, then 
$u_i \leq v_i$ in $\R^N\times[0,T)$ for $i=1,2$. 
\end{Prop}
\begin{proof}
For $\mu$, $\delta>0$, define 
\[
\theta_{\mu,\delta}
:=\sup_{i,x,t}\{u_i(x,t) -v_i(x,t) - \frac{\mu}{T-t} -\delta|x|^2\}. 
\]
By the assumption~\eqref{eq:prop4.2}, we see that 
$\theta_{\mu,\delta}\leq 2e^{-\nu t}R\leq 2R$. 
Put $\displaystyle \theta:=\limsup_{\mu,\delta\to0}\theta_{\mu,\delta}$. 

If $\theta\leq0$, for all $x$, $t$, $i$, we have 
\[
u_i(x,t) - v_i (x,t) \leq \frac{\mu}{T-t} + \delta|x|^2 + \theta_{\mu,\delta}. 
\]
Taking $\limsup$, we then have $u_i(x,t) - v_i(x,t)\leq0$. 

To obtain a contradiction, suppose that $\theta>0$. 
There exists a subsequence (expressed by the same symbol) such that 
$\theta_{\mu,\delta}\geq \theta/2>0$. 
In what follows fix $\mu$ and $\delta$ sufficiently small. 
Consider the doubling of the variables 
\[
(i,x,y,t,s)\mapsto 
u_i(x,t)-v_i(y,s) -\frac{\mu}{T-t} -\delta|x|^2 
-\frac{|x-y|^2}{2\e} - \frac{|t-s|^2}{2\e}, 
\]
where $\e\in(0,1)$ is a parameter. 
Assume that the doubling map attains a maximum at 
$(i_\e,x_\e,y_\e,t_\e,s_\e)$. 
We may assume $(i_\e,x_\e,y_\e,t_\e,s_\e)
\in\{1,2\}\times\overline{B_{r_\delta}}^2\times[0,T-\tau_\mu]^2$ for some $r_\delta>0$ and $\tau_\mu>0$, where $B_r$ stand for the ball centered at the origin with radius $r>0$.

It follows from 
\[
\theta_{\mu,\delta} \leq 
u_i(x,t)-v_i(y,s) -\frac{\mu}{T-t} -\delta|x|^2 
-\frac{|x-y|^2}{2\e} - \frac{|t-s|^2}{2\e} 
\]
at $(i_\e,x_\e,y_\e,t_\e,s_\e)$ 
that 
\begin{equation}\label{eq:theta}
\begin{aligned}
\frac{|x_\e - y_\e|^2}{2\e} + \frac{|t_\e - s_\e|^2}{2\e} 
+\frac{\mu}{T- t_\e} +\delta|x_\e|^2 
&\leq
u_{i_\e}(x_\e,t_\e)-v_{i_\e}(y_\e,s_\e) -\theta_{\mu,\delta} \\
&\leq 4R. 
\end{aligned}
\end{equation}
Taking a subsequence of $i_\e$ such that $i_\e \equiv \hat{i}\in\{1,2\}$ for 
sufficiently $\e\ll 1$. 
Take a subsequence if necessary, we find 
$(\hat{x},\hat{t},\hat{i}) 
\in \overline{B_{r_\delta}}\times[0,T-\tau_\mu)\times \{1,2\}$ 
such that 
\[
x_\e, y_\e \to \hat{x}, \quad t_\e, s_\e\to \hat{t}, \quad i_\e \to\hat{i}
\]
as $\e\to0$. 
It follows from~\eqref{eq:theta} that 
\[
\begin{aligned}
&\limsup_{\e\to 0}
\left(
 \frac{|x_\e -y_\e|^2}{2\e} +  \frac{|t_\e - s_\e|^2}{2\e} 
 \right)\\
& \leq 
\limsup_{\e\to 0}
\left(
u_{\hat{i}}(x_\e,t_\e)-v_{\hat{i}}(y_\e,_\e) 
-\frac{\mu}{T-t_\e} -\delta|x_\e|^2
\right)
 -\theta_{\mu,\delta}\\
&\leq  
u_{\hat{i}}(\hat{x},\hat{t}) - v_{\hat{i}}(\hat{y},\hat{s}) 
-\frac{\mu}{T-\hat{t}} -\delta|\hat{x}|^2 -\theta_{\mu,\delta} \\
&\leq0. 
 \end{aligned}
\]
Thus, we obtain 
\[
\lim_{\e\to0} \frac{|x_\e -y_\e|^2}{2\e}
=\lim_{\e\to0} \frac{|t_\e - s_\e|^2}{2\e} 
= 0. 
\]
To obtain a contradiction, suppose that $\hat{t}=0$. 
Then 
\[
\begin{aligned}
0&\leq
\limsup_{\e\to 0}
\left(
u_{\hat{i}}(x_\e,t_\e)-v_{\hat{i}}(y_\e,_\e) 
-\frac{\mu}{T-t_\e} -\delta|x_\e|^2
\right)
 -\theta_{\mu,\delta}\\
 &= u_{\hat{i}}(\hat{x},0) - v_{\hat{i}}(\hat{x},0) - \frac{\mu}{T} -\delta|\hat{x}|^2
 -\theta_{\mu,\delta}\\
 &\leq0. 
\end{aligned}
\]
Therefore, we obtain 
\[
u_{\hat{i}}(\hat{x},0) - v_{\hat{i}}(\hat{x},0) - \frac{\mu}{T} -\delta|\hat{x}|^2
= \theta_{\mu,\delta}. 
\]
which is impossible. 
Therefore, $\hat{t}$ must be positive.

By the Ishii lemma (see \cite[Lemma~2.3.23]{IS}, \cite[Chapter~3]{G}), it holds that 
\[
\begin{aligned}
\left(
\frac{t_\e -s_\e}{\e}, \frac{x_\e -y_\e}{\e}, X
\right)
&\in
\overline{PJ}^{2,+}
\left(
u_{\hat{i}}(x_\e,t_\e) -\delta|x_\e|^2 - \frac{\mu}{T-t_\e}
\right), \\
\left(
\frac{t_\e -s_\e}{\e}, \frac{x_\e -y_\e}{\e}, Y
\right)
&\in
\overline{PJ}^{2,-} v_{\hat{i}}(y_\e,s_\e)
\end{aligned}
\]
and 
\begin{equation*}
\begin{aligned}
-\frac{3}{\e}
\begin{pmatrix}
I & O\\
O & I
\end{pmatrix}
\leq 
\begin{pmatrix}
X & O\\
O & -Y
\end{pmatrix}
\leq 
\frac{3}{\e}
\begin{pmatrix}
I & -I\\
-I & I
\end{pmatrix}, 
\end{aligned}
\end{equation*}
where $\overline{PJ}^{2,+}$ are defined in~\eqref{eq:jet}. 
In general, 
\[
\begin{aligned}
\left(
\frac{t_\e -s_\e}{\e}, \frac{x_\e -y_\e}{\e}, X
\right)
&\in
\overline{PJ}^{2,+}
\left(
u_{\hat{i}}(x_\e,t_\e) -\delta|x_\e|^2 - \frac{\mu}{T-t_\e}
\right)
\end{aligned}
\]
is equivalent to
\[
\begin{aligned}
\left(
\frac{t_\e -s_\e}{\e} +  \frac{\mu}{(T-t_\e)^2}, 
\frac{x_\e -y_\e}{\e} +2\delta x_\e, 
X+2\delta I
\right)
&\in
\overline{PJ}^{2,+}
u_{\hat{i}}(x_\e,t_\e). 
\end{aligned}
\]
We now set $i:= \hat{i}$, $p_1=p$, $p_2=q$ and $j$ denotes $j\neq i$. 
In what following ,we drop the subscription $\e$ for simplisity. 
Since $(u_1,u_2)$ is a viscosity subsolution to~\eqref{eq:lem4.1} and 
$(v_1,v_2)$ is a viscosity supersolution to~\eqref{eq:lem4.1}, 
we have
\[
\begin{aligned}
\frac{t -s}{\e} +  \frac{\mu}{(T-t)^2}
+ F_{i} (x, X+ 2\delta I) 
+ \nu u_i(x,t) - e^{(p_i -1)\nu t} |u_j|^{p_i -1}u_j  \leq 0, \\
\frac{t -s}{\e} 
+ F_{i} (y, Y ) 
+ \nu v_i(y,s) - e^{(p_i -1)\nu t} |v_j|^{p_i -1}v_j  \geq 0. \\
\end{aligned}
\]
Combining these inequalities, we see that 
\begin{equation}\label{eq:lem.est1}
\begin{aligned}
&\frac{\mu}{(T-t)^2} + \nu(u_i(x,t)-v_i(y,s))
- e^{(p_i -1)\nu t} (|u_j|^{p_i -1}u_j -|v_j|^{p_i -1}v_j )\\
&\leq F_i(y,Y) -F_i (x,X+2\delta I). 
\end{aligned}
\end{equation}
By~Lemma~\ref{lem:structure},  
there exists a modulus of continuity $\omega_{F_i}:[0,\infty)\to\R$ 
such that 
\[
F_i (y,Y) - F_i(x,X)
\leq 
\omega_{F_i} 
\left(
|x-y| + \frac{|x-y|^2}{\e}
 \right) 
\]
for all $x,y\in\R^N$. 
This together with~\eqref{eq:ellipticity} implies that 
\begin{equation}\label{eq:lem.est2}
\begin{aligned}
& F_i(y,Y) -F_i (x,X+2\delta I) + F_i(x,X) - F_i(x,X) \\
&\leq 
\omega_{F_i}
\left(
|x-y| + \frac{|x-y|^2}{\e} 
\right)
+\Pu_i^+(-2\delta I)\\
&= 
\omega_{F_i}
\left(
|x-y| + \frac{|x-y|^2}{\e} 
\right)
+
2\delta \Lambda_i N .
\end{aligned}
\end{equation}
If $u_j(x,t) \leq v_j(y,s)$, then 
\begin{equation}\label{eq:lem.est3}
\nu(u_i -v_i) -e^{(p_i -1)\nu t} (|u_j|^{p_i -1}u_j -|v_j|^{p_i -1}v_j )
\geq 
\lambda(u_i -v_i)\geq
\lambda\frac{\theta}{2}>0. 
\end{equation}
On the other hand, if $u_j(x,t) > v_j(y,s)$, 
it follows from the mean value theorem 
and~\eqref{eq:prop4.2} 
 that 
\begin{equation}\label{eq:R}
\begin{aligned}
e^{(p_i -1)\nu t} (|u_j|^{p_i -1}u_j -|v_j|^{p_i -1}v_j )
&=
e^{(p_i -1)\nu t} 
\frac{|u_j|^{p_i -1}u_j -|v_j|^{p_i -1}v_j }{u_j -v_j}(u_j -v_j) \\
&\leq 
 R^{p_i -1}(u_j -v_j). 
\end{aligned}
\end{equation}
It follows from~\eqref{eq:lem.est3} and~\eqref{eq:R} that 
\[
\nu(u_i -v_i) -e^{(p_i -1)\nu t} (|u_j|^{p_i -1}u_j -|v_j|^{p_i -1}v_j )
\geq 
\nu(u_i -v_i) -  R^{p_i -1}(u_j -v_j). 
\]
Choose $\nu>0$ so large that 
\[
\nu > 
\max
\left\{
R^{p_1 -1}, R^{p_2 -1}
\right\} 
\]
to get 
\begin{equation}\label{eq:lem.est4}
\nu(u_i -v_i) -  R^{p_i -1}(u_j -v_j)
\geq 
(\nu -R^{p_i -1})(u_i-v_i)
\geq (\nu -R^{p_i -1}) \frac{\theta}{2} >0. 
\end{equation}

Summarazing~\eqref{eq:lem.est1}--\eqref{eq:lem.est3}, \eqref{eq:lem.est4}, 
we obtain 
\[
\frac{\mu}{(T-t)^2} + (\nu -R^{p_i -1}) \frac{\theta}{2}
\leq 
\omega_{F_i}
\left(
|x-y| + \frac{|x-y|^2}{\e} 
\right)
+
2\delta \Lambda_i N .
\]
Dropping the first term of the left-hand and then passing to the limit $\e\to0$, 
we finally obtain 
\[
(\nu -R^{p_i -1}) \frac{\theta}{2}
\leq 2\delta \Lambda_i N .
\]
Since $\delta$ is sufficiently small, this is a contradiction. 
\end{proof}

\appendix
\section{Perron's method}
In this appendix, we state Perron's method and give its proof. 
Let $T>0$ and let $S$ be the set of all viscosity solutions of \eqref{eq:P} in $\R^N\times(0,T)$. 
\begin{Lemma}\label{lem:B1}
Let $T>0$. 
Assume that $S\neq\emptyset$. 
For $i=1,2$, set 
\[
u_i (x,t):= \sup\{v_i(x,t) | \  v=(v_1,v_2) \in S\}, \quad x\in\R^N, t\in(0,T).  
\]
If 
\begin{equation*}
\sup_K |u_i| < \infty ,\quad i=1,2
\end{equation*}
for any compact subset $K\subset \R^N \times (0,T)$, 
then $u=(u_1,u_2)$ is a viscosity subsolution of \eqref{eq:P}.  
\end{Lemma}
\begin{proof}
For $i=1,2$ and $\vphi\in C^{2,1}(\R^N\times(0,T))$, we assume that $u_i^*-\vphi$ attains strict maximum at $(x_0,t_0)\in \R^N \times(0,T)$. 
Choose $r>0$ so that $[t_0-r,t_0+r]\subset (0,T)$. 
Then for any $\e>0$, we have
\[
u_i^*(x_0,t_0) 
= \lim_{\e\to0} \sup_{\substack{x\in B(x_0,\e) \\ |t-t_0|<\e}} u_i(x,t)
\leq
\sup_{\substack{x\in B(x_0,\e) \\ |t-t_0|<\e}} u_i(x,t), 
\]
where $u_i^*$ us defined in~\eqref{usce}. 
For all $\tau>0$, there exist sequences $x_{\tau,\e}\in B(x_0,\e)$ and 
$t_{\tau,\e}\in (t_0-r, t_0+r)$ such that 
\[
\sup_{\substack{x\in B(x_0,\e) \\ |t-t_0|<\e}} u_i(x,t) 
\leq u_i(x_{\tau,\e},t_{\tau,\e}) +\tau. 
\]
For any $\tau>0$ there exists $\delta_\tau >0$ such that for all $x\in\R^N$, $t\in(0,T)$, if 
$|x-x_0| + |t-t_0|<\delta_\tau$, then $|\vphi(x,t) - \vphi(x_0,t_0)|<\tau$. 
For each $k=1,2,\dots$ set 
\[
\e=\min
\left\{
\frac{\delta_{1/k}}{2}, \frac{1}{k},r
\right\}. 
\]
There exist $x_k \in B_r(x_0)$, $t_k\in(t_0-r,t+r)$ such that 
\[
\begin{aligned}
&x_k \to x_0, \quad t_k\to t_0 \quad \mathrm{as} \ k\to\infty,\\
&u_i^*(x_0,t_0) < u_i(x_k,t_k) + \frac{1}{k}, \quad 
 |\vphi(x_k,t_k)- \vphi(x_0,t_0)| <\frac{1}{k}.  
\end{aligned}
\]
Moreover, by the  definition of $u_i$, there exists $(u_1^k,u_2^k)\in S$ such that 
\[
u_i(x_k,t_k) < u_i^k(x_k,t_k) +\frac{1}{k}. 
\]
Choose $(y_k,s_k)\in\overline{B}_r(x_0)\times[t_0-r,t_0+r]$ 
so that  $(u_i^k)^* -\vphi$ attains its maximum at $(y_k,s_k)$. 
Taking a subsequence (still denoted by the same symbol), we see that as $k\to\infty$, 
$y_k \to \hat{y}$, $t_k\to\hat{t}$ 
for some $\hat{y}\in \overline{B}_r(x_0)$, $t_k\to\hat{t}\in[t_0-r,t_0+r]$. 
We have 
\begin{equation}\label{eq:limsup}
\begin{aligned}
(u_i^*-\vphi)(x_0,t_0) 
& < (u_i^*-\vphi)(x_k,t_k) +\frac{3}{k}\\
&\leq ((u_i^k)^*-\vphi)(x_k,t_k) +\frac{3}{k}\\
&\leq ((u_i^k)^*-\vphi)(y_k,s_k) +\frac{3}{k}\\
&\leq (u_i^*-\vphi)(y_k,s_k) +\frac{3}{k}. \\
\end{aligned}
\end{equation}
Since $u_i^*$ is lower semicontinuous, we see that 
\[
(u_i^* -\vphi)(x_0,t_0) \leq (u_i^* - \vphi)(\hat{y},\hat{t}). 
\]
On the other hand, $u_i^* -\vphi$ has a strict maximum at $(x_0,t_0)$, hence
$x_0=\hat{x}$ and $t_0 =\hat{t}$. 
Therefore, $y_k\in \overline{B}_r(x_0)$, $t_k\in [t_0-r,t_0+r]$ for sufficiently large $k$. 
In addition, using~\eqref{eq:limsup} again, we have
\[
\lim_{k\to\infty}(u_i^k)^*(y_k,s_k) = u_i^*(x_0,t_0) 
\]
and 
\[
\begin{aligned}
\limsup_{k\to\infty} (u_i^k)^*(y_k,s_k) 
\leq
\limsup_{k\to\infty}u_j^*(y_k,s_k)
\leq u_j^*(x_0,t_0), \\
\end{aligned}
\]
where $j\neq i$. 
Consequently, 
we obtain 
\[
\partial_t \vphi(y_k,s_k) + F_i(y_k,D^2\vphi(y_k,s_k)) 
\leq|u_j^*(y_k,s_k)|^{p_i -1}u_j^*(y_k,s_k), 
\]
hence 
\[
\partial_t \vphi(x_0,t_0) + F_i(x_0,D^2\vphi(x_0,t_0)) 
\leq|u_j^*(x_0,t_0)|^{p_i -1}u_j^*(x_0,t_0), 
\]
which completes the proof. 
\end{proof}

\begin{Prop}
Assume that $\xi=(\xi_1,\xi_2)\in (L^{\infty}_{\mathrm{loc}}(\R^N\times(0,T)))^2$ 
is a viscosity subsolution of \eqref{eq:P} 
and $\eta=(\eta_1,\eta_2)\in (L^{\infty}_{\mathrm{loc}}(\R^N\times(0,T)))^2$ 
is a viscosity supersolution of \eqref{eq:P} for some $T>0$. 
If $\xi$ and $\eta$ satisfy 
\begin{equation*}
\xi_1 \leq \eta_1, \quad \xi_2 \leq \eta_2 \quad \mathrm{in} \ \R^N \times(0,T), 
\end{equation*}
then 
\[
u_i(x,t) := \sup\{v_i(x,t) | v=(v_1,v_2)\in S, \ \xi \leq v \leq \eta\}
\]
is a viscosity solution of \eqref{eq:P} in $\R^N\times(0,T)$. 
Here 
$\xi \leq v \leq \eta$ means that $\xi_i \leq v_i \leq \eta_i$ in $\R^N\times(0,T)$ for $i=1,2$. 
\end{Prop}
\begin{proof}
By Lemma~\ref{lem:B1}, $u=(u_1,u_2)$ is a viscosity solution to~\eqref{eq:P}. 
Suppose, contrary to our claim, that 
there exist $\vphi \in C^{2,1}$ and $(x_0,t_0)\in \R^N\times(0,T)$,  $u_{i*} -\vphi$ attains a strict minimum at $(x_0,t_0)$, 
$(u_{i*}-\vphi)(x_0,t_0) =0$ for some $i=1,2$ and exists $\theta>0$ such that 
\begin{equation}\label{eq:lem:B1}
\partial_t \vphi + F(x_0,D^2\vphi) - |u_{i*}|^{p_i -1}u_{i*} < -\theta 
\end{equation}
at $(x_0,t_0)$, 
where $i\neq j$ and $p_1 =p$, $p_2 =q$. 

We firstly show that $\vphi(x_0,t_0) <\eta_{i*}(x_0,t_0)$. 
In fact, we see that $\vphi \leq u_{i*} \leq \eta_{i*}$ and $u_{j*} \leq \eta_{j*}$
and $\eta_{i*} - \vphi$ attains a minimum at $(x_0,t_0)$ 
if $\vphi(x_0,t_0) = \eta_{i*}(x_0,t_0)$, thus by the definition of the viscosity supersolution, we obtain 
\[
\partial_t \vphi(x_0,t_0) + F(x_0,D^2\vphi(x_0,t_0))
\geq |\eta_{j*}|^{p_j - 1}\eta_{j*} (x_0,t_0)
\geq |u_{j*}|^{p_j -1}u_{j*}(x_0,t_0). 
\]
This contradicts the assumption. 

For any $\rho>0$, there exists $\e_\rho>0$ such that 
\[
\begin{aligned}
u_{j*}(x_0,t_0) -\rho 
&=
\sup_{\e>0}\inf_{\substack{|x-x_0|<\e_\rho\\|t-t_0|<\e_\rho}} u_j(x,t) -\rho \\
&< \inf_{\substack{|x-x_0|<\e_\rho \\|t-t_0|<\e_\rho}} u_j(x,t)\\
&\leq u_j(x,t)
\end{aligned}
\]
for $|x-x_0|<\e_\rho$ and $|t-t_0|<\e_\rho$. 
By using the mean value theorem, there exists $\hat{\theta}$ such that 
\[
\begin{aligned}
|u_{j*}^{p_i -1}u_j 
&\geq
|u_{j*}(x_0,t_0) -\rho|^{p_i -1}(u_{j*}(x_0,t_0)-\rho) \\
&=
|u_{j*}(x_0,t_0) |^{p_i -1}u_{j*}(x_0,t_0) \\
&- \rho p_i 
|
\hat{\theta} u_{j*}(x_0,t_0)(x_0,t_0) 
+ (1-\hat{\theta})(u_{j*}(x_0,t_0)-\rho)
|^{p_i -1}.  
\end{aligned}
\]
We can find $s_0>0$ so small that 
\[
\rho p_i 
|
\hat{\theta} u_{j*}(x_0,t_0)(x_0,t_0) 
+ (1-\hat{\theta})(u_{j*}(x_0,t_0)-\rho)
|^{p_i -1} < \frac{\theta}{4}
\]
and 
\[
\vert 
\partial_t\vphi(x_0,t_0) + F_i(x_0,D^2\vphi(x_0,t_0)) 
- \partial_t\vphi(x,t) - F_i(x,D^2\vphi(x,t)) 
\vert 
<\frac{\theta}{4}
\]
for $|x-x_0|<s_0$ and $|t-t_0|<s_0$. 
This together with \eqref{eq:lem:B1} and the continuity of $\partial_t\vphi$ and $F_i(\cdot,D^2\vphi)$ implies that 
\[
\begin{aligned}
-\theta 
&> 
\partial_t \vphi (x,t) +F_i(x,D^2\vphi(x,t)) -\frac{\theta}{4}
-|u_{j*}|^{p_i-1}u_{j*}(x,t) -\frac{\theta}{4} 
\end{aligned}
\]
for $|x-x_0| < s_0$ and $|t-t_0|<s_0$. 
Therefore, 
\begin{equation}\label{eq:phi}
\partial_t \vphi (x,t) +F_i(x,D^2\vphi(x,t)) 
-|u_{j*}|^{p_i-1}u_{j*}(x,t) 
< -\frac{\theta}{2}. 
\end{equation}
It is already shown that $u_{i*}(x_0,t_0)=\vphi(x_0,t_0) <\eta_{i*}(x_0,t_0)$. 
Set 
\[
3\hat{\tau} :=\eta_{i*}(x_0,t_0) -u_{i*}(x_0,t_0)>0. 
\] 
Since $\eta_{i*}$ is lower semicontinuous and $\vphi$ is continuous, we can find $s_1\in(0,s_0)$ 
such that for all $|x-x_0| <s_1$ and $t\in(t_0 -s_1,t_0+s_1)$, 
\[
\eta_{i*}(x,t) -\vphi(x,t)
> 
\eta_{i*}(x_0,t_0) -\vphi(x_0,t_0) - \hat{\tau} = 2 \hat{\tau}. 
\]
Therefore $\vphi(x,t) +2\hat{\tau} < \eta_{i*}(x,t))$ in $D$, where 
\[
D:= B_{s_1}(x_0) \times(t_0-s_1,t_0+s_1). 
\]
On the other hand, since $u_{i*}-\vphi$ attains a strict minimum at $(x_0,t_0)$ and $(u_{i*}-\vphi)(x_0,t_0)=0$, there exist $\e\in(0,s_1 /2)$ and $\tau_0 \in(0,\hat{\tau})$ such that 
\[
(u_{i*}-\vphi)(x,t) \geq \min_{(x,t)\in A}\{ (u_{i*}-\vphi)(x,t)\} > \tau_0. 
\] 
Here we have set 
\[
A := 
\left(
\overline{B_{s_1 /2 +\e}(x_0)}\setminus B_{s_1 /2 -\e}(x_0)
\right)
\times 
\left\{
t \mid \frac{s_1}{2}-\e \leq |t-t_0| \leq \frac{s_1}{2} +\e
\right\}.
\]
We now define $(w_1,w_2)$ by 
\[
\begin{aligned}
w_i(x,t) &:=
\begin{cases}
\max\{u_i(x,t),\vphi(x,t) + \tau_0\} \quad &\mathrm{in} \ D/2,\\
u_i(x,t) \quad \mathrm{in} \ (\R^N\times(0,T)) \setminus (D/2),
\end{cases}
\\
w_j (x,t) &:= u_j(x,t) \quad \mathrm{in} \ \R^N\times(0,T), 
\end{aligned}
\]
where 
\[
D/2 := B_{s_1 /2}(x_0) \times \left(t_0- \frac{s_1}{2},t_0+ \frac{s_1}{2} \right). 
\]
In what follows, we shall show that $(w_1,w_2)$ is a viscosity subsolution to~\eqref{eq:P} in $\R^N\times(0,T)$ satisfying $\xi_k \leq w_k \leq \eta_k$ for $k=1,2$. 
It follows from the definition of $w_i$ that we have $\xi_i \leq u_i \leq w_i$ in 
$\R^N\times(0,T)$. 
Since $\vphi(x,t) + \tau_0 \leq \eta_{i*}$ in $D$, we see that $\vphi + \tau_0 \leq \eta_i$ in $D$, 
hence $w_k \leq \eta_k$ in $\R^N\times(0,T)$ for $k=1,2$. 
Consequently, for $k=1,2$, we obtain 
\[
\xi_k \leq w_k \leq \eta_k \quad \mathrm{in} \ \R^N\times(0,T). 
\]
We can find $n\in\{1,2,\dots\}$ sufficiently large so that
\[
\frac{1}{n} < \frac{s_1}{2} -\e \quad \mathrm{and} \quad 
\frac{1}{n} < \frac{\tau_0}{2} 
\]
and there exist $x_n \in B_{1/n}(x_0)$ and $t_n\in\R$ with $|t_0-t_n|<1/n$ 
such that 
\[
u_i (x_n,t_n) 
< u_{i*}(x_0,t_0) + \frac{1}{n}. 
\]
Moreover, it follows that 
\[
 u_{i*}(x_0,t_0) + \frac{1}{n}
 <u_{i*}(x_0,t_0) + \frac{\tau_0}{2} 
 = \vphi(x_0,t_0) + \frac{\tau_0}{2}
 < \vphi(x_0,t_0) + \tau_0. 
\]
Note that $(x_n,t_n) \in D/2$. 

In what follows, we shall prove that $(w_1,w_2)$ is a viscosity subsolution to~\eqref{eq:P} in $\R^n\times(0,T)$.  
Let us take $\R^N\times(0,T)$ and $\psi\in C^{2,1}(\R^N\times(0,T))$ arbitrarily. 

We firstly assume that 
$w_{i}^* - \psi$ attains a local maximum at $(\hat{x},\hat{t})$. 
Consider the first case 
$w_i^* (\hat{x},\hat{t}) = u_i^* (\hat{x},\hat{t})$. 
Then 
\[
\begin{aligned}
u_i^* (\hat{x},\hat{t}) -\psi(\hat{x},\hat{t}) 
&= w_i^*(\hat{x},\hat{t}) - \psi (\hat{x},\hat{t}) \\
&\geq w_i^* (x,t) -\psi(x,t)\\
&\geq u_i^* (x,t) - \psi(x,t)
\end{aligned}
\]
in $\R^N\times(0,T)$. 
Thus, $u_i^* -\psi$ attains its maximum at $(\hat{x},\hat{t})$. 
Moreover, since $(u_1,u_2)$ is a subsolution to~\eqref{eq:P} in $\R^N\times(0,T)$ and $u_j \equiv w_j$, we have
\begin{equation*}
\partial_t \psi + F_i(\cdot,D^2\psi) 
\leq |u_j^*|^{p_i -1}u_j^* 
=|w_j^*|^{p_i -1}w_j^*
\end{equation*}
at $(\hat{x},\hat{t})$. 

We next consider the second case 
$w_i^*(\hat{x},\hat{t}) =(\vphi +\tau_0)^*(\hat{x},\hat{t}) =\vphi(\hat{x},\hat{t})+\tau_0$. 
Note that $(\hat{x},\hat{t})\in D/2$. 
The same argument above implies that $\vphi+\tau_0-\psi$ attains its maximum at $(\hat{x},\hat{t})$. 
Thus, we see that 
\[
\partial_t \vphi(\hat{x},\hat{t}) = \partial_t \psi(\hat{x},\hat{t}), \quad
D\vphi(\hat{x},\hat{t}) = D\psi(\hat{x},\hat{t}), \quad 
D^2 \vphi(\hat{x},\hat{t}) \leq D^2 \psi(\hat{x},\hat{t}). 
\]
It follows from~\eqref{eq:ellipticity} and~\eqref{eq:phi} that 
\[
\begin{aligned}
\partial_t  \psi + F_i(\cdot,D^2\psi) 
&\leq 
\partial_t  \vphi + F_i(\cdot,D^2\vphi) \\
&\leq |{u_j}_*|^{p_i -1} {u_j}_* \\
&\leq |{u_j}^*|^{p_i -1} {u_j}^* \\
&= |{w_j}^*|^{p_i -1} {w_j}^*
\end{aligned}
\]
at $(\hat{x},\hat{t})$. 

We secondly assume that 
$w_{j}^* - \psi$ attains a local maximum at $(\hat{x},\hat{t})$. 
Since $w_j =u_j$ in $\R^N\times(0,T)$, 
$u_j^*-\psi $ also attains its maximum at $(\hat{x},\hat{t})$. 
Therefore, we obtain
\[
\begin{aligned}
\partial_t \psi + F_j(\cdot,D^2\psi) 
\leq |u_i^*|^{p_j -1}u_i^*
\leq |w_i^*|^{p_j -1}w_i^* 
\end{aligned}
\] 
at $(\hat{x},\hat{t})$. 

Consequently, $(w_1,w_2)$ is a viscosity subsotlution to~\eqref{eq:P} in $\R^N\times(0,T)$. 
This contradicts the definition of $(u_1,u_2)$. 
\end{proof}
\begin{Remark}
Let $f_1$ and $f_2$ be a real valued function defined in a subset of $\R^M$ with $M\in\{1,2,\dots\}$. 
Then $ \max\{f_1,f_2\}^* = \max\{f_1^*,f_2^*\}$. 
This fact allows us to divide the cases $u_i^* = w_i^*$ or not. 
\end{Remark}
%

\section*{Acknoledgement}
TK was partially supported by Grant-in-Aid for Early-Career Scientists JSPS KAKENHI Grant Number 18K13436 and Tottori University of Environmental Studies Grant-in-Aid for Special Research. 
RS was partially supported by Grant-in-Aid for Early-Career Scientists JSPS KAKENHI Grant Number 18K13435 and funding from Fukuoka University (Grant No. 205004).


%


\begin{thebibliography}{10}
\bibitem{AT}
S. N.  Armstrong and M. Trokhimtchouk, Long-time asymptotics for fully nonlinear homogeneous parabolic equations, Calc. Var. Partial Differential Equations \textbf{38}  (2010), 521--540. 



\bibitem{CL}
M. G. Crandall,  P.-L.  Lions, Quadratic growth of solutions of fully nonlinear second order equations in $R^n{}$, 
 Differential Integral Equations \textbf{3} (1990), 601--616. 

 
\bibitem{DL}
K. Deng, H. A. Levine, The role of critical exponents in blow-up theorems: the sequel,  J. Math. Anal. Appl.  \textbf{243}  (2000), 85--126. 
 

\bibitem{EH1}
M. Escobedo and M. A. Herrero, 
Boundedness and blow up for a semilinear reaction-diffusion system, 
JDE \textbf{89}, (1991), 176--202. 


\bibitem{FI1} Y. Fujishima, K. Ishige, 
Blowing up solutions for nonlinear parabolic systems with unequal elliptic operators, 
 J. Dynam. Differential Equations  \textbf{32}  (2020), 1219--1231. 

\bibitem{FI2}  Y. Fujishima, K. Ishige, 
Initial traces and solvability of Cauchy problem to a semilinear parabolic system, 
 J. Math. Soc. Japan \textbf{73}  (2021), 1187--1219. 


\bibitem{F}
H. Fujita, On the blowing up of solutions of the Cauchy problem for 
$u_{t}=\bigtriangleup u+ u^{1+\alpha}$,  J. Fac. Sci. Univ. Tokyo Sect. I \textbf{13} 109--124 (1966). 



\bibitem{G} 
Y. Giga, Surface evolution equations, 
A level set approach, 
Monographs in Mathematics, {\bf 99} Birkh\"{a}user Verlag, Basel,  2006. 

\bibitem{HV}
Y. Huang and J.L. V'{a}zquez, Large-time geometrical properties of solutions of the Barenblatt equation of elasto-plastic filtration, 
 J. Differential Equations  \textbf{252}  (2012), 4229--4242. 

\bibitem{IS}
C. Imbert and L. Silvestre, 
An introduction to fully nonlinear parabolic equations, 
An introduction to the K\"{a}hler-Ricci flow, 
Lecture Notes in Math. \textbf{2086} Springer, Cham, 2013, 7--88.

\bibitem{KPV}
S. Kamin,  L. A.  Peletier, J. L. V\'{a}zquez, On the Barenblatt equation of elastoplastic filtration, 
 Indiana Univ. Math. J.  \textbf{40}  (1991), 1333--1362. 

\bibitem{Ko}
S. Koike,
A Beginner’s Guide to the Theory of Viscosity Solutions,
MSJ memoir \textbf{13}, Math. Soc. Japan, 2004. 

\bibitem{K}
N. V. Krylov, 
Sobolev and viscosity solutions for fully nonlinear elliptic and parabolic equations, 
Mathematical Surveys and Monographs, \textbf{233},
American Mathematical Society, Providence, RI,  2018. 



\bibitem{L}
G. M. Lieberman, 
Second order parabolic differential equations, 
World Scientific Publishing Co., Inc., River Edge, NJ,  1996. 

\bibitem{Me} 
P. Meier, On the critical exponent for reaction-diffusion equations, 
 Arch. Rational Mech. Anal. \textbf{109}  (1990), 63--71. 

\bibitem{MQ1}
R. Meneses and  A.  Quaas, Fujita type exponent for fully nonlinear parabolic equations and existence results, 
 J. Math. Anal. Appl. \textbf{376}  (2011), 514--527. 
 
\bibitem{MQ2}
R. Meneses and  A. Quaas, Existence and non-existence of global solutions for uniformly parabolic equations, J. Evol. Equ. \textbf{12}  (2012), 943--955. 



\bibitem{U}
Y. Uda, The critical exponent for a weakly coupled system of the generalized Fujita type reaction-diffusion equations, 
 Z. Angew. Math. Phys. \textbf{46} (1995), 366--383. 

\bibitem{QS1}
P. Quittner and Ph. Souplet, Superlinear parabolic problems, 
Blow-up, global existence and steady states, 
Second edition, Birkh\"{a}user/Springer, Cham,  2019. 


\bibitem{W}
L. Wang, 
On the regularity theory of fully nonlinear parabolic equations: II,
Comm. Pure Appl. Math. \textbf{45}  (1992), 141-178. 





\end{thebibliography}
\end{document}